\documentclass[12pt]{amsart}

\usepackage{amssymb, latexsym, euscript}
\usepackage{amssymb}
\usepackage{latexsym}
\usepackage{amsmath}

\def\sD{{\mathfrak D}}      
   \def\sH{{\mathfrak H}}   
   \def\sK{{\mathfrak K}}   \def\sL{{\mathfrak L}}
\def\sM{{\mathfrak M}}      
      
   \def\sT{{\mathfrak T}}

      \def\dC{{\mathbb C}}
      
      \def\dI{{\mathbb I}}
      
   \def\dN{{\mathbb N}}   
      \def\dR{{\mathbb R}}
   \def\dT{{\mathbb T}}   \def\dU{{\mathbb U}}

\def\cD{{\mathcal D}}      \def\cF{{\mathcal F}}
\def\cG{{\mathcal G}}   \def\cH{{\mathcal H}}   
\def\cJ{{\mathcal J}}   \def\cK{{\mathcal K}}   \def\cL{{\mathcal L}}
\def\cM{{\mathcal M}}   \def\cN{{\mathcal N}}   
      \def\cR{{\mathcal R}}
   \def\cT{{\mathcal T}}

   \def\bB{{\mathbf B}}

\def\fL{{\mathsf L}}

\def\CM{{\cM}}
\def\RE{{\rm Re\,}}
\def\IM{{\rm Im\,}}
\def\ran{{\rm ran\,}}
\def\cran{{\rm \overline{ran}\,}}
\def\dom{{\rm dom\,}}
\def\mul{{\rm mul\,}}
\def\cdom{{\rm \overline{dom}\,}}
\def\span{{\rm span\,}}
\def\cspan{{\rm \overline{span}\, }}
\def\cmr{{\dC \backslash \dR}}
\def\uphar{{\upharpoonright\,}}
\def\RE{{\rm Re\,}}
\def\IM{{\rm Im\,}}

\def\wt{\widetilde}
\def\wh{\widehat}

\def\f{\varphi}

\newtheorem{theorem}{Theorem}[section]

\newtheorem{proposition}[theorem]{Proposition}
\newtheorem{corollary}[theorem]{Corollary}
\newtheorem{definition}[theorem]{Definition}
\newtheorem{remark}[theorem]{Remark}

\numberwithin{equation}{section}

\numberwithin{equation}{section}
\begin{document}

\title [Transformations of Nevanlinna operator-functions]
{Transformations of Nevanlinna operator-functions and their fixed points}
\author[Yu.M.~Arlinski\u{\i}]{Yu.M. Arlinski\u{\i}}
\address{
Department of Mathematics, Dragomanov National Pedagogical University,
Kiev, Pirogova 9, 01601, Ukraine} \email{yury.arlinskii@gmail.com}

\dedicatory{To Eduard R.~Tsekanovski\u{\i} on the occasion of his 80-th birthday}


\vskip 1truecm
\thispagestyle{empty}
\baselineskip=12pt
\begin{abstract}
  We give a new characterization of the class ${\bf N}^0_{\mathfrak M}[-1,1]$ of the operator-valued in the Hilbert space ${\mathfrak M}$  Nevanlinna functions that admit representations as compressed resolvents ($m$-functions) of selfadjoint contractions. We consider the automorphism ${\bf \Gamma}:$ $M(\lambda){\mapsto}M_{{\bf \Gamma}}(\lambda):=\left((\lambda^2-1)M(\lambda)\right)^{-1}$ of the class ${\bf N}^0_{\mathfrak M}[-1,1]$ and construct a realization of $M_{{\bf \Gamma}}(\lambda)$ as a compressed resolvent. The unique fixed point of ${\bf\Gamma}$ is the $m$-function of the block-operator Jacobi matrix related to the Chebyshev polynomials of the first kind. We study a transformation ${\bf\wh \Gamma}:$ ${\mathcal M}(\lambda)\mapsto {\mathcal M}_{{\bf\widehat \Gamma}}(\lambda) :=-({\mathcal M}(\lambda)+\lambda I_{\mathfrak M})^{-1}$ that maps the set of all Nevanlinna operator-valued functions into its subset. The unique fixed point $\cM_0$ of ${\bf\wh \Gamma}$ admits a realization as the compressed resolvent of the "free" discrete Schr\"{o}dinger operator ${\bf\widehat J}_0$ in the Hilbert space ${\bf H}_0=\ell^2(\dN_0)\bigotimes\sM$. We prove that ${\mathcal M}_0$ is the uniform limit on compact sets of the open upper/lower half-plane in the operator norm topology of the iterations $\{{\mathcal M}_{n+1}(\lambda)=-({\mathcal M}_n(\lambda)+\lambda I_\mathfrak M)^{-1}\}$ of ${\bf\widehat\Gamma}$. We show that the pair $\{{\bf H}_0,{\bf \widehat J}_0\}$ is the inductive limit of the sequence of realizations $\{\wh{\mathfrak H}_n,\wh A_n\}$ of $\{{\mathcal M}_n\}$. In the scalar case $({\mathfrak M}={\mathbb C})$, applying the algorithm of I.S.~Kac, a realization of iterates $\{{\mathcal M}_n\}$ as $m$-functions of canonical (Hamiltonian) systems is constructed.
\end{abstract}
\maketitle
\section{Introduction and preliminaries}
{\bf Notations.} We use the symbols $\dom T$, $\ran T$, $\ker T$ for
the domain, the range, and the null-subspace of a linear operator
$T$. The closures of $\dom T$, $\ran T$ are denoted by $\cdom T$,
$\cran T$, respectively. The identity operator in a Hilbert space
$\sH$ is denoted by  $I$ and sometimes by $I_\sH$. If $\sL$ is a
subspace, i.e., a closed linear subset of $\sH$, the orthogonal
projection in $\sH$ onto $\sL$ is denoted by $P_\sL.$ The notation
$T\uphar \sL$ means the restriction of a linear operator $T$ on the
set $\sL\subset\dom T$. The resolvent set of $T$ is denoted by
$\rho(T)$. The linear space of bounded
operators acting between Hilbert spaces $\sH$ and $\sK$ is denoted
by $\bB(\sH,\sK)$ and the Banach algebra $\bB(\sH,\sH)$ by
$\bB(\sH).$ Throughout
this paper we consider separable Hilbert spaces over the field $\dC$
of complex numbers. $\dC_{+}/\dC_-$ denotes the open upper/lower half-plane of $\dC$, $\dR_+:=[0,+\infty)$,
$\dN$ is the set of natural numbers, $\dN_0:=\dN\cup\{0\}$.
\vskip 0.3cm
\begin{definition}
\label{nfu}
A $\bB(\sM)$-valued function $M$ is  called a Nevanlinna
function ($R$-function \cite{KacK}, \cite{Shmul71}, Herglotz function \cite{GesTsek2000}, Herglotz-Nevanlinna function \cite{ArlBelTsek2011}, \cite{ArlKlotz2010}) if it is holomorphic outside the real axis, symmetric
$M(\lambda)^*=M(\bar\lambda)$, and satisfies the inequality $\IM
\lambda\, \IM M(\lambda)\ge 0$ for all $\lambda\in\cmr$.
\end{definition}
This class
is often denoted by $\cR[\sM]$.
A more general is the notion of Nevanlinna family, cf. \cite{DHMS06}.
\begin{definition}\label{Nevanlinna}
A family of linear relations $\CM(\lambda)$, $\lambda \in \cmr$, in a
Hilbert space $\sM$ is called a \textit{Nevanlinna family} if:
\begin{enumerate}
\item $\CM(\lambda)$ is maximal dissipative for every $\lambda \in \dC_+$
(resp. accumulative for every $\lambda \in \dC_-$);
\item $\CM(\lambda)^*=\cM(\bar \lambda)$, $\lambda \in \cmr$;
\item for some, and hence for all, $\mu \in \dC_+ (\dC_-)$ the
      operator family $(\CM(\lambda)+\mu I_\sM)^{-1} (\in \bB(\sM))$ is
      holomorphic on $\dC_+ (\dC_-)$.
\end{enumerate}
\end{definition}
 The class of all Nevanlinna families in a Hilbert space $\sM$ is denoted by $\wt R(\sM)$.
Each Nevanlinna familiy $\CM\in\wt R(\sM)$ admits the
following decomposition to the operator part $M_s(\lambda)$,
$\lambda\in\cmr$, and constant multi-valued part $M_\infty$:
\[
  \CM(\lambda)=M_s(\lambda)\oplus M_\infty,\quad
  M_\infty=\{0\}\times\mul \CM(\lambda).
\]
Here $M_s(\lambda)$ is a Nevanlinna family of densely defined
operators in $\sM \ominus \mul \CM(\lambda)$.

A Nevanlinna $\bB(\sM)$-valued function admits the
integral representation, see \cite{KacK}, \cite{Shmul71},
\begin{equation}
\label{INTrep}
 M(\lambda)
  =A+B\lambda+\int_{\dR}\left(\frac{1}{t-\lambda}-\frac{t}{t^2+1}\right)\,
               d\Sigma(t),
 \quad
\int_{\dR}\,\frac{d\Sigma(t)}{t^2+1}\in\bB(\sM),
\end{equation}
where $A=A^*\in\bB(\sM)$, $0\le B=B^*\in\bB(\sM)$, the
$\bB(\sM)$-valued function $\Sigma(\cdot)$ is nondecreasing and $\Sigma(t)=\Sigma(t-0)$. The
integral is uniformly convergent in the strong topology; cf.
\cite{Br}, \cite{KacK}.
The following condition is equivalent to the definition of a $\bB(\sM)$-valued Nevanlinna function $M(\lambda)$ holomorphic on $\cmr$ : the function of two variables
\[
K(\lambda,\mu)=\cfrac{M(\lambda)-M(\mu)^*}{\lambda-\bar \mu}
\]
is a nonnegative kernel, i.e., $\sum\limits_{k,l=1}^n\left( K(\lambda_k,\lambda_l) f_l, f_k\right)\ge 0$
for an arbitrary set of points \\
$\{\lambda_1,\lambda_2,\ldots,\lambda_n\}\subset\dC_+/(\subset\dC_-)$ and an arbitrary set of vectors $\{f_1,f_2,\ldots, f_n\}\subset\sM$.

It follows from \eqref{INTrep} that
\[
B=s-\lim\limits_{y\uparrow\infty}\cfrac{M(iy)}{y}=s-\lim\limits_{y\uparrow\infty}\cfrac{\IM M(iy)}{y},
\]
\[
 \IM M(iy)=B\,y+\int_{\dR}\frac{y}{t^2+y^2}\,d\Sigma(t),
\]
and this implies that $\lim_{y\to \infty} y\IM M(iy)$ exists in the strong resolvent sense as a selfadjoint relation; see e.g.
\cite{BHSW2010}. This limit is a bounded selfadjoint operator if and
only if $B=0$ and $\int_{\dR}\,d\Sigma(t)\in \bB(\sM)$, in which
case
 $s-\lim_{y\to \infty} y\IM M(iy)=\int_{\dR}\,d\Sigma(t).
$ 
In this case one can rewrite the integral representation
\eqref{INTrep} in the form
\begin{equation}
\label{INTrep0}
 M(\lambda)=  E+\int_{\dR}\frac{1}{t-\lambda}\,d\Sigma(t), \quad
 \int_{\dR}\,d\Sigma(t)\in \bB(\sM),
\end{equation}
and $E=\lim_{y\to \infty} M(iy)$ in
$\bB(\sM)$.

The class of $\bB(\sM)$-valued Nevanlinna
functions $M$ with the integral representation \eqref{INTrep0} with
$E=0$ is denoted by $\cR_0[\sM]$.
In this paper we will consider the following subclasses of the class $\cR_0[\sM]$.
\begin{definition}
\label{Nev}
A function $N$ from the class $\cR_0[\sM]$ is said to belong to the class
\begin{enumerate}
\item ${\cN}[\sM]$ if
$s-\lim_{y\to \infty}iy N(iy)=-I_\sM,
$
\item ${{\bf N}}_{\sM}^0$ if $N\in {\cN}[\sM]$ and $N$ is holomorphic at infinity,
\item ${\bf N}_{\sM}^0[-1,1]$ if $N\in {{\bf N}}^{0}_{\sM}$  and is holomorphic outside the interval $[-1,1].$
\end{enumerate}
\end{definition}
Thus, we have inclusions
\[
{\bf N}_{\sM}^0[-1,1]\subset {{\bf N}}_{\sM}^0\subset{\cN}[\sM]\subset \cR_0[\sM]\subset \cR[\sM]\subset \wt R(\sM).
\]

A selafdjoint operator $T$ in the Hilbert space $\sH$ is called \textit{$\sM$-simple}, where $\sM$ is a subspace of $\sH$, if
$
\cspan\{T-\lambda I)^{-1}\sM,\;\lambda\in\dC_+\cup\dC_-\}=\sH.
$ 
If $T$ is bounded then the latter condition is equivalent to $\cspan\{T^n\sM,\; n\in\dN_0\}=\sH.$

The next theorem follows from \cite[Theorem 4.8]{Br} and the Na\u{\i}mark's dilation theorem \cite[Theorem 1, Appendix
I]{Br}, see \cite{AHS2} and \cite{ArlKlotz2010} for the case $M\in{\bf N}^0_\sM$.

\begin{theorem}
\label{BROD} 1) If $M\in {\cN}[\sM]$,
 then there exist a Hilbert space $\sH$ containing $\sM$ as a
subspace and a selfadjoint operator $T$ in $\sH$ such that $T$ is
$\sM$-simple and
\begin{equation}
\label{rep2}
 M(\lambda)=P_\sM (T-\lambda I)^{-1}\uphar\sM.
\end{equation}
for $\lambda$ in the domain of $M$.
If $M\in{\bf N}^0_\sM$, then $T$ is bounded and if $M\in{\bf N}^0_\sM[-1,1],$ then $T$ is a selfadjoint contraction.

2) If $T_1$ and $T_2$ are selfadjoint operators in the Hilbert spaces $\sH_1$ and $\sH_2$, respectively, $\sM$ is a subspace in $\sH_1$ and $\sH_2$, $T_1$ and $T_2$ are $\sM$-simple, and
\[
M(\lambda)=P_\sM (T_1-\lambda I_{\sH_1})^{-1}\uphar\sM= P_\sM (T_2-\lambda I_{\sH_2})^{-1}\uphar\sM,\; \lambda\in\cmr,
\]
then there exists a unitary operator $U$ mapping $\sH_1$ onto $\sH_2$ such that
\[
U\uphar\sM=I_\sM\quad\mbox{and}\quad UT_1=T_2U.
\]
\end{theorem}
The right hand side in \eqref{rep2} is often called  \textit{compressed resolvent/$\sM$-resolvent/the Weyl function/$m$-function}, \cite{Ber} \cite{GS}.
A representation $M\in {\bf N}^0_\sM$ in the form \eqref{rep2} will be called a \textit{realization of $M$}.

We show in Section \ref{teil2}, that
$M(\lambda)\in{\bf N}^0_\sM[-1,1]  \Longleftrightarrow(\lambda^2-1)^{-1} M(\lambda)^{-1}\in{\bf N}^0_\sM[-1,1].$
It follows that the transformation
\begin{equation}\label{GTR1}
{\bf N}_{\sM}^0[-1,1]\ni M(\lambda)\stackrel{{{\bf \Gamma}}}{\mapsto}M_{{\bf \Gamma}}(\lambda):=\cfrac{M(\lambda)^{-1}}{\lambda^2-1}\in {\bf N}_{\sM}^0[-1,1]
\end{equation}
maps the class ${\bf N}_{\sM}^0[-1,1]$ onto itself and ${{\bf \Gamma}}^{-1}={{\bf \Gamma}}$. In Theorem \ref{ghtj} we construct a realization
of $(\lambda^2-1)^{-1}M(\lambda)^{-1}$ as a compressed resolvent by means of the contraction $T$ that realizes $M$.
 The mapping ${\bf \Gamma}$ has the unique fixed point $M_0(\lambda)=-\cfrac{I_\sM}{\sqrt{\lambda^2-1}}$ that is compressed resolvent $P_{\sM_0}({\bf J}_0-\lambda I)^{-1}\uphar\sM_0 $ of the block-operator Jacobi matrix \begin{equation}\label{chebb1}
{\bf J}_0=\begin{bmatrix} 0 & \cfrac{1}{\sqrt{2}}\,I_\sM & 0 &0   & 0 &
\cdot &
\cdot &\cdot \\
\cfrac{1}{\sqrt{2}}\,I_\sM & 0 & \cfrac{1}{{2}}\,I_\sM & 0 &0& \cdot &
\cdot &\cdot \\
0    & \cfrac{1}{{2}}\,I_\sM & 0 & \cfrac{1}{{2}}\,I_\sM &0& \cdot &
\cdot&\cdot   \\
0    & 0& \cfrac{1}{{2}}\,I_\sM & 0 & \cfrac{1}{{2}} \,I_\sM&0& \cdot &
\cdot
\\
\vdots & \vdots & \vdots & \vdots & \vdots & \vdots & \vdots&\vdots
\end{bmatrix},
\end{equation}
acting in the Hilbert space $\ell^2(\dN_0)\bigotimes\sM,$ and $\sM_0=\sM\oplus\{0\}\oplus\cdots$, see Proposition \ref{fixpoint}.

A selfadjoint linear relation $\wt A$ in the orthogonal sum $\sM\oplus\cK$ is called \textit{minimal with respect to $\sM$}
   (see \cite[page 5366]{DHMS06}) if
\[
\sM\oplus\cK=\cspan\left\{\sM+(\wt A-\lambda I)^{-1}\sM: \lambda\in\rho(\wt A)\right\}.
\]

One of the statements obtained in \cite{DHMS06} in the context of the Weyl family of a boundary relation is the following:
\begin{theorem}\label{readhms}
Let $\cM$ be a Nevanlinna family in the
Hilbert space $\sM$. Then there exists unique up to unitary
equivalence
a selfadjoint linear relation $\wt A$ in the Hilbert space $\sM\oplus\cK$ such that $\wt A$ is minimal with respect to $\sM$ and the equality
\begin{equation}
\label{opexpr1}
\cM(\lambda)=-\left(P_\sM\left(\wt A-\lambda I\right)^{-1}\uphar\sM\right)^{-1}-\lambda I_\sM ,\; \lambda\in\cmr
\end{equation}
holds.
\end{theorem}
The equivalent form of \eqref{opexpr1} is
$$P_\sM(\wt A-\lambda I)^{-1}\uphar\sM=-(\cM(\lambda)+\lambda I_\sM)^{-1},\;\lambda\in\cmr.$$
The compressed resolvent $P_\sM(\wt A-\lambda I)^{-1}\uphar\sM$ belongs to the class $\cR_0[\sM]$ and even to its more narrow subclass, see Corollary \ref{ccorr1}.

In Section \ref{ffixx} we consider the following mapping defined on the whole class $\wt R(\sM)$ of Nevanlinna families:
\begin{equation}
\label{denken}
\cM(\lambda)\stackrel{{{\bf {\wh\Gamma}}}}{\mapsto}\cM_{{\bf {\wh\Gamma}}}(\lambda):=-(\cM(\lambda)+\lambda I_\sM)^{-1}, \; \lambda\in\cmr.
\end{equation}
We prove (Theorem \ref{fixpunkt}) that the mapping ${\bf{\wh \Gamma}}$ and each its degree ${\bf{\wh \Gamma}}^k$ has the unique fixed point
$$\cM_0(\lambda)=\cfrac{-\lambda+\sqrt{\lambda^2-4}}{2}\,I_\sM$$
and the sequence of iterations
\[
\cM_1(\lambda)=-(\cM(\lambda)+\lambda I_\sM)^{-1},\; \cM_{n+1}(\lambda)=-(\cM_n(\lambda)+\lambda I_\sM)^{-1},\;n\in\dN,
\]
starting with an arbitrary Nevanlinna family $\cM$, converges to $\cM_0$ in the operator norm topology uniformly on compact sets lying in the open left/right half-plane of the complex plane. The function $\cM_0(\lambda)$ can be realized by the free discrete Schr\"{o}dinger operator given by the block-operator Jacobi matrix
\begin{equation}\label{clschope}
{\bf\wh J_0}=\begin{bmatrix} 0 & I_\sM & 0 &0   & 0 &
\cdot &
\cdot &\cdot \\
I_\sM & 0 & I_\sM & 0 &0& \cdot &
\cdot &\cdot \\
0    & I_\sM & 0 & I_\sM &0& \cdot &
\cdot&\cdot   \\
\vdots & \vdots & \vdots & \vdots & \vdots & \vdots & \vdots&\vdots
\end{bmatrix}
\end{equation}
acting in the Hilbert space $\ell^2(\dN_0)\bigotimes\sM.$
Besides we construct a sequence  $\{\wh\sH_n,\wh A_n\}$ of realizations of functions $\cM_n$ ($\cM_{n}(\lambda)=P_\sM(\wh A_{n-1}-\lambda I)^{-1}\uphar\sM,$ $\lambda\in\cmr$) and show that the  Hilbert space $\ell^2(\dN_0)\bigotimes\sM$ and the block-operator Jacobi matrix ${\bf\wh J_0}$ are the inductive limits of
$\{\wh\sH_n\}$ and $\{\wh A_n\}$, respectively.
Observe that when $\sM=\dC,$ the Jacobi matrices ${\bf J}_0$ and $\cfrac{1}{2}{\bf\wh J}_0$ are connected with Chebyshev polynomials of the first and second kinds, respectively \cite{Ber}.

Let $\cH(t)=\begin{bmatrix}h_{11}(t)&h_{12}(t)\cr h_{21}(t)&h_{22}(t) \end{bmatrix}$ be symmetric and nonnegative $2\times 2$ matrix-function with scalar real-valued entries on $\dR_+$. Assume that $\cH(t)$ is locally integrable on $\dR_+$ and is \textit{trace-normed}, i.e., ${\rm tr}\,\cH(t)=1$ a.e. on $\dR_+.$
Let $\cJ=\begin{bmatrix} 0&-1\cr 1&0\end{bmatrix}$. The system of differential equations
\begin{equation}\label{cansys}
\cJ\cfrac{d\vec x}{dt}=\lambda\cH(t)\vec x(t), \; \vec x(t)=\begin{bmatrix}x_1(t)\cr x_2(t)\end{bmatrix},\; t\in \dR_+,\; \lambda\in\dC,
\end{equation}
is called \textit{ the canonical system with  the Hamiltonian $\cH$} or \textit{the Hamiltonian system}.

The \textit{$m$-function} $m_\cH$ of the canonical system \eqref{cansys} can be defined as follows:
\[
m_\cH(\lambda)=\cfrac{x_2(0,\lambda)}{x_1(0,\lambda)},\; \lambda\in\cmr,
\]
where $\vec x(t,\lambda)$ is the solution of \eqref{cansys}, satisfying
\[
x_1(0,\lambda)\ne 0\quad\mbox{and}\quad\int\limits_{\dR_+}\vec x(t,\lambda)^*\cH(t)\vec x(t,\lambda)dt<\infty.
\]
The $m$-function of a canonical system is a Nevanlinna function. As has been proved by L.~de Branges \cite{deBr}, see also \cite{W}, for each Nevanlinna function $m$ there exists a unique trace-normed  canonical system such that its $m$-function $m_\cH$ coincides with $m.$
In the last Section \ref{LSEC}, applying the algorithm suggested by I.S.~Kac in \cite{Kac1999}, we construct a sequence of Hamiltonians $\{\cH_n\}$ such that the $m$-functions of the corresponding canonical systems coincides with the sequence of the iterates $\{m_n\}$ of the mapping ${\bf\wh\Gamma}$
$$m_1(\lambda)=-\cfrac{1}{m(\lambda)+\lambda}\,,\ldots, m_{n+1}(\lambda)=-\cfrac{1}{m_n(\lambda)+\lambda}\,,\ldots,\;\lambda\in\cmr,$$
where $m(\lambda)$ is a non-rational Nevanlinna function form the class ${\bf N}^0_\dC$. This sequence $\{m_n\}$ converges locally uniformly on $\dC_+/\dC_-$
to the function $m_0(\lambda)=\cfrac{-\lambda+\sqrt{\lambda^2-4}}{2}$ that is the $m$-function of the canonical system with the Hamiltonian
$$\cH_0(t)=\begin{bmatrix}\cos^2(j+1)\cfrac{\pi}{2}&0\cr 0&\sin^2(j+1)\cfrac{\pi}{2}\end{bmatrix},\; t\in [j,j+1)\;\forall j\in\dN_0.$$
For the constructed Hamiltonian $\cH_n$ the property $\cH_n\uphar[0,n+1)=\cH_0\uphar [0,n+1)$ is valid for each $n\in\dN$.
Moreover, our construction shows that for the Hamiltonian $\cH$ such that the $m$-function $m_\cH$ of the corresponding canonical system belongs to the class ${\bf N}^0_\dC$, the Hamiltonian $\cH_{\bf\wh\Gamma}$ of the canonical system having ${\bf\wh\Gamma}(m)$ as its $m$-function, is of the form
\[
\cH_{\bf\wh\Gamma}(t)=\left\{\begin{array}{l}\cH_0(t),\; t\in[0,2)\\
\begin{bmatrix}1&0\cr 0&1\end{bmatrix}-\cH(t-1),\; t\in[2,+\infty) \end{array}\right..
\]

\section{Characterizations of subclasses}\label{teil2}
\subsection{The subclass $\cR_0[\sM]$}
The next proposition is well known, cf.\cite{Br}.
\begin{proposition}
\label{opexpr12}  Let $M(\lambda)$ be a $\bB(\sM)$-valued Nevanlinna
function. Then the following statements are equivalent:
\begin{enumerate}
\def\labelenumi{\rm (\roman{enumi})}
\item $M\in R_0[\sM]$;

\item the function $y\|M(iy)\|$ is bounded on $[1,\infty)$,
\item there exists a strong limit
$s-\lim\limits_{y\to+\infty} iy M(iy)=-C,
$ 
where $C$ is a bounded selfadjoint nonnegative operator in $\sM$;
\item $M$ admits a representation
\begin{equation}
\label{drugghtl}
 M(\lambda)=K^*(T-\lambda I)^{-1}K,\quad \lambda\in\cmr,
\end{equation}
where $T$ is a selfadjoint operator in a Hilbert space $\cK$ and $K\in\bB(\sM,\cK)$; here $\cK$, $T$, and $K$ can be
selected such that $T$ is $\cran K$-simple, i.e.,
$$\cspan\{(T-\lambda)^{-1}\ran
K:\;\lambda\in\cmr\}=\cK.$$

\end{enumerate}
\end{proposition}

\begin{proposition}\label{yjdd} \cite[Lemma 2.14, Example 6.6]{DHMS06}.
Let $\cK$ and $\sM$ be Hilbert spaces, let $K\in\bB(\sM,\cK)$ and let $D$ and $T$ be selfadjoint operators in $\sM$ and $\cK$, respectively. Consider a selfadjoint operator $\wt A$ in the Hilbert space $\sM\oplus\cK$ given by the block-operator matrix
\[
\wt A=\begin{bmatrix} D&K^*\cr K& T\end{bmatrix},\; \dom\wt A=\dom D\oplus\dom T.
\]
Then $\wt A$ is $\sM$-minimal if and only if  $T$ is $\cran K$-simple.
\end{proposition}
\begin{proof}
Our proof is based on the Schur-Frobenius formula for the resolvent $(\wt A-\lambda I)^{-1}$
\begin{multline}
\label{Sh-Fr1}
(\wt A-\lambda I)^{-1}=
\begin{bmatrix}-V(\lambda)^{-1}&V(\lambda)^{-1}K^*(T-\lambda I)^{-1}\cr
(T-\lambda I)^{-1}KV(\lambda)^{-1}&(T-\lambda I)^{-1}\left(I_\cK-KV(\lambda)^{-1}K^*(T-\lambda I)^{-1}\right)
\end{bmatrix},\\
V(\lambda):=\lambda I_\sM-D+K^*(T-\lambda I)^{-1}K,\;
\lambda\in\rho(T)\cap\rho(\wt A).
\end{multline}
Actually, \eqref{Sh-Fr1} implies the equivalences
\begin{multline*}
\cspan\left\{\sM+(\wt A-\lambda I)^{-1}\sM: \lambda\in\cmr\right\}
=\sM\oplus\cK\\
\Longleftrightarrow\cK\bigcap\limits_{\lambda\in\cmr}\ker
\left(P_\sM(\wt A-\lambda I)^{-1}\right)=\{0\}\Longleftrightarrow\bigcap\limits_{\lambda\in\cmr}\ker\left(K^*(T-\lambda I)^{-1}\right)=\{0\}\\
\Longleftrightarrow\cspan\{(T-\lambda)^{-1}\ran K:\;\lambda\in\cmr\}=\cK.
\end{multline*}
\end{proof}

In the sequel we will use the following consequence of \eqref{Sh-Fr1}:
\begin{equation} \label{Sh-Fr2}
P_\sM(\wt A-\lambda I)^{-1}\uphar\sM=-\left(-D+K^*(T-\lambda I_\sM)^{-1} K+\lambda I_\sM\right)^{-1},\; \lambda\in\rho(T)\cap\rho(\wt A).
\end{equation}
\begin{proposition}\label{opexpr12b} cf. \cite[the proof of Theorem 3.9]{DHMS06}. For a $\bB(\sM)$-valued Nevanlinna function $M$ the following
statements are equivalent:
\begin{enumerate}
\def\labelenumi{\rm (\roman{enumi})}
\item the limit value
 $C:=-s-\lim\limits_{y\to+\infty}iy M(iy)$
satisfies $ 0\leq C \leq I_\sM$;
\item $M$ admits a representation
\begin{equation}
\label{drugght2} M(\lambda)=P_\sM(\wt A-\lambda I)^{-1}\uphar
\sM,\;\lambda\in\cmr,
\end{equation}
where $\wt A$ is a selfadjoint linear relation in a Hilbert space $\sH\supset
\sM$ and $P_\sM$ is the orthogonal projection from $\sH$ onto $\sM$;
\item $M$ admits a representation \eqref{drugghtl}
with a contraction $K\in\bB(\sM,\wt\sH)$;

\item the following inequality holds
\[
 \frac{\IM M(\lambda)}{\IM \lambda}-M(\lambda)M(\lambda)^*\geq 0,\quad \lambda\in\cmr.
\]
\end{enumerate}
In (ii) $\sH$ and $\wt A$ can be selected such that $\wt A$ is minimal w.r.t. $\sM.$
 Moreover,
$\wt A$ in \eqref{drugght2} can be taken to be a selfadjoint operator if
and only if $C=I_\sM$.
The operator $K$ in (iii) is an isometry if and only if $C=I_\sM$.
\end{proposition}
\begin{proof}
The equivalence (i)$\Longleftrightarrow$ (iii) follows from Proposition \ref{opexpr12}.

(i)$\Longrightarrow$(iv)

 Since \eqref{drugghtl} holds, we get $C=K^*K$ and the inequality $0\le C\le I_\sM$ implies $||K||\le 1$ and, therefore holds the inequality.
\[
 \frac{\IM M(\lambda)}{\IM \lambda}-M(\lambda)M(\lambda)^*\geq 0,\quad \lambda\in\cmr.
 \]
(iv)$\Longrightarrow$(ii)

Consider $-M(\lambda)^{-1}$. Then
\[
\cfrac{\IM(- M(\lambda)^{-1}h-\lambda h,h)}{\IM\lambda}=\cfrac{\IM(- M(\lambda)^{-1}h,h)}{\IM\lambda}-||h||^2\ge 0, \;h\in\sM.
\]
Hence $\cM(\lambda):=-M(\lambda)^{-1}-\lambda I_\sM$ is a Nevanlinna family.
Due to Theorem \ref{readhms} and \eqref{opexpr1}
we have
\[
-(\cM(\lambda)+\lambda I_\sM)^{-1}=P_\sM(\wt A-\lambda I_\sH)^{-1}\uphar\sM,\;\lambda\in\cmr,
\]
where $\wt A$ is a selfadjoint linear relation in some Hilbert space $\sH=\sM\oplus\cK$.

(ii)$\Longrightarrow$(i)

Let $\wh A_0$ be the operator part of $\wt A$ acting in a subspace $\sH_0$ of $\sH.$ Decompose $\wt A$ as $H={\rm Gr}\wh A_0\oplus\{0,\sH\ominus\sH_0\}$.
Then
\[
P_\sM(\wt A-\lambda I)^{-1}\uphar\sM=P_{\sM}(\wh A_0-\lambda I)^{-1}P_{\sH_0}\uphar\sM
=P_{\sM}P_{\sH_0}(\wh A_0-\lambda I)^{-1}P_{\sH_0}\uphar\sM.
\]
Set $K=P_{\sH_0}\uphar\sM:\sM\to\sH_0$. Then $K^*=P_\sM P_{\sH_0}$, $||K||\le 1,$
\[
M(\lambda)=K^*(\wh A_0-\lambda I)^{-1}K,\;\lambda\in\cmr,
\]
and
\[
s-\lim\limits_{x\to+\infty}iyM(iy)=-K^*K,\; C=K^*K\in [0,I_\sM].
\]
(iii)$\Longrightarrow$(ii)

Since $||K||\le 1$, $\cM(\lambda)=-M^{-1}(\lambda)-\lambda I_\sM$ is a Nevanlinna family. By Theorem \ref{readhms} there is a Hilbert space $\cK$ and a selfadjoint linear relation $\wt A$ in $\sM\oplus\cK$ minimal w.r.t. $\sM$ such that
$\cM(\lambda)=-\left(P_\sM(\wt A-\lambda I)^{-1}\uphar\sM\right)^{-1}-\lambda I_\sM,$ $\lambda\in\cmr.$

\end{proof}

\begin{corollary}
\label{ccorr1} There is a one-to-one correspondence between all
Nevanlinna families $\cM$ in $\sM$ and all
$\bB(\sM)$-valued Nevanlinna functions $M$ satisfying the
condition (ii) in Proposition \ref{opexpr12} with $C\in[0, I_\sM]$. This
correspondence is given by the relations
\[
M(\lambda)=-(\cM(\lambda)+\lambda I_\sM)^{-1},\; \cM(\lambda)=-M(\lambda)^{-1}-\lambda I_\sM,\; \lambda\in\cmr.
\]
\end{corollary}
\begin{remark} \label{ComDil}
 For the case $\sM=\dC$ the statement of Corollary \ref{ccorr1} can be found in \cite[Chapter VII, $\S$1, Lemma 1.7]{Ber}.

In \cite{DLS} (see also \cite{ADW}) it is established that an $\bB(\sM)$-valued function $M(\lambda)$, $\lambda\in\cD\subset\dC_+/\dC_-$ admits the representation
\eqref{drugght2} iff the kernel
\[
K(\lambda,\mu)=\cfrac{M(\lambda)-M(\mu)^*}{\lambda-\bar \mu}-M(\mu)^*M(\lambda)
\]
is nonnegative on $\cD.$
\end{remark}

\subsection{The subclass ${\bf N}^{0}_\sM [-1,1]$}
Notice, that if $M\in{\bf N}^{0}_\sM [-1,1]$, then
\[
\left\{\begin{array}{l} (M(x)g,g)>0 \;\forall g\in\sM\setminus\{0\},\; x<-1,\\
(M(x)g,g)<0 \;\forall g\in\sM\setminus\{0\},\; x>1 \end{array}\right..
\]
Therefore, see \cite[Appendix]{KrNu}
$$(1+\lambda)M(\lambda),\; (1-\lambda)M(\lambda)\in\cR[\sM].$$
\begin{theorem} \label{ghtj}
1) A $\bB(\sM)$-valued Nevanlinna function $M$ belongs to ${\bf N}_{\sM}^0[-1,1]$ if and only if the function
\[
\fL(\lambda,\xi)=\frac{(1-\lambda^2)M(\lambda)-(1-\bar\xi^2)M(\xi)^*
-(\lambda-\bar\xi) I_\sM}{\lambda-\bar\xi},
\]
with $\lambda,\xi\in\dC\setminus[-1,1],$ $\lambda\ne\bar{\xi}$ is a nonnegative kernel.

2) If $M\in {\bf N}_{\sM}^0[-1,1]$, then the function
\[
\cfrac{M(\lambda)^{-1}}{\lambda^2-1},\;\lambda\in \dC\setminus[-1,1]
\]
belongs to ${\bf N}_{\sM}^0[-1,1]$ as well.

3) If a selfadjoint contraction $T$ in the Hilbert space $\sH$, containing $\sM$ as a subspace, realizes $M$, i.e.,
 $M(\lambda)=P_\sM(T-\lambda I)^{-1}\uphar\sM$, for all $\lambda\in\dC\setminus[-1,1]$, then
\[
\cfrac{M(\lambda)^{-1}}{\lambda^2-1}=P_\sM({\bf T}-\lambda I)^{-1}\uphar\sM,\; \lambda\in\dC\setminus [-1,1],
\]
where a selfadjoint contraction ${\bf T}$ is given by
\begin{equation}
\label{wollen}
{\bf T}:=\begin{bmatrix}-P_\sM T\uphar\sM&P_\sM D_T\cr D_T\uphar\sM &T \end{bmatrix}:\begin{array}{l}\sM\\\oplus\\\sD_T\end{array}\to\begin{array}{l}\sM\\\oplus\\\sD_T\end{array},
\end{equation}
and $D_T:=(I-T^2)^{1/2}$, $\sD_T:=\cran D_T$.
Moreover,
if $T$ is $\sM$-simple, then ${\bf T}$ is $\sM$-simple as well and
the operator ${\bf T}\uphar\sD_{\bf T}$ is unitarily equivalent to the operator $P_{{\sM}^\perp} T\uphar{\sM}^\perp.$
\end{theorem}
\begin{proof}  The statement in 1) follows from \cite[Theorem 6.1]{AHS2}. Observe that if $M(\lambda)=P_\sM(T-\lambda I)^{-1}\uphar\sM$ $\forall \lambda\in\dC\setminus[-1,1]$,  where $T$ is a selfadjoint contraction, then
\begin{multline}\label{zlhjl}
\fL(\lambda,\xi)=\frac{(1-\lambda^2)M(\lambda)-(1-\bar\xi^2)M(\xi)^*
-(\lambda-\bar\xi) I_\sM}{\lambda-\bar\xi}\\
=P_\sM (T-\lambda I)^{-1}(I-T^2)(T-\overline{\xi}I)^{-1}\uphar\sM,\; \lambda,\xi\in\dC\setminus[-1,1],\;\lambda\ne\bar{\xi}.
\end{multline}

2) Let $\lambda\in\dC\setminus [-1,1]$, then
\[
\left|\left((T-\lambda I) h, h\right)\right|\ge d(\lambda)||h||^2  \;\forall h\in\sH,
\]
where $d(\lambda)={\rm {dist}}(\lambda,[-1,1])$.
Set $h=(T-\lambda I)^{-1}f$, $f \in\sM$. Then
\begin{multline*}
||M(\lambda)f||||f||\ge\left|(f,M(\lambda)f)\right|=\left|\left(f,(T-\lambda I)^{-1}f\right)\right|\\
=\left|\left(h,(T-\lambda I) h\right)\right|
\ge d(\lambda)||h||^2\ge c(\lambda)||f||^2,\; c(\lambda)>0.
\end{multline*}
Hence, $||M(\lambda)f||\ge c(\lambda)||f||$ and since $M(\bar\lambda)=M(\lambda)^*$, we get  $||M(\lambda)^*f||\ge c(\bar\lambda)||f||.$
It follows that $M(\lambda)^{-1}\in\bB(\sM)$ for all $\lambda\in\dC\setminus [-1,1]$.

 Set
\[
L(\lambda):=(1-\lambda^2)M(\lambda)-\lambda I_\sM, \;\lambda\in\dC\setminus [-1,1].
\]
Then from \eqref{zlhjl} we get
\begin{multline*}
L(\lambda)-L(\lambda)^*=(1-\lambda^2)M(\lambda)-(1-\bar\lambda^2)(M(\lambda)^*-(\lambda-\bar\lambda)I_\sM\\
=(\lambda-\bar\lambda)P_\sM (T-\lambda I)^{-1}(I-T^2)(T-\bar{\lambda}I)^{-1}\uphar\sM.
\end{multline*}
It follows that $L(\lambda)$ and the functions
\[
(1-\lambda^2)M(\lambda)=L(\lambda)+\lambda I_\sM,\;\lambda\in\dC\setminus[-1,1]
\]
and
\[
-\left((1-\lambda^2)M(\lambda)\right)^{-1}=\cfrac{M(\lambda)^{-1}}{\lambda^2-1},\; \lambda\in\dC\setminus[-1,1]
\]
are Nevanlinna functions. Then from the equality $M(\lambda)=-\lambda^{-1}+o(\lambda^{-1}),$ $\lambda\to\infty$, we get that also
 \[
 \cfrac{M(\lambda)^{-1}}{\lambda^2-1}= -\lambda^{-1}+o(\lambda^{-1}),\;\lambda\to\infty,
 \]
i.e.,
 \[
\cfrac{M(\lambda)^{-1}}{\lambda^2-1}\in{\bf N}^0_\sM[-1,1].
\]
3) Observe that the subspace $\sD_T$ is contained in the Hilbert space $\sH.$
Let ${\bf H}:=\sM\oplus\sD_T$ and let
${\bf T}$ be given by \eqref{wollen}. Since $T$ is a selfadjoint contraction in $\sH$, we get for an arbitrary $\varphi\in \sM$ and $f\in\sD_T$ the equalities
\[
\left(\begin{bmatrix}\varphi\cr f\end{bmatrix}, \begin{bmatrix}\varphi\cr f\end{bmatrix}\right)\pm
\left(\begin{bmatrix}\varphi\cr f\end{bmatrix},{\bf  T}\begin{bmatrix}\varphi\cr f\end{bmatrix}\right)=
\left\|(I\mp T)^{1/2}\f\pm (I\pm T)^{1/2}f \right\|^2.
\]
Therefore ${\bf T}$ is a selfadjoint contraction in the Hilbert space ${\bf H}.$

Applying \eqref{Sh-Fr2} we obtain
\begin{multline*}
P_\sM({\bf T}-\lambda I)^{-1}\uphar\sM=-\left(\lambda I+P_\sM T\uphar\sM+P_\sM D_T (T-\lambda I)^{-1}D_T\uphar\sM\right)^{-1}\\
=-\left(\lambda I+P_\sM\left(T(T-\lambda I)+I-T^2\right)(T-\lambda I)^{-1}\uphar\sM\right)^{-1}\\
=-\left(\lambda I+P_\sM(I-\lambda T)(T-\lambda I)^{-1}\uphar\sM\right)^{-1}=-\left((1-\lambda^2)P_\sM(T-\lambda I)^{-1}\uphar\sM\right)^{-1}\\
=\cfrac{M^{-1}(\lambda)}{\lambda^2-1},\; \lambda\in \dC\setminus [-1,1].
\end{multline*}
Suppose that $T$ is $\sM$-simple, i.e.,
$$\cspan\{T^n\sM,\;n\in\dN_0\}=\sM\oplus\cK\Longleftrightarrow\bigcap\limits_{n=0}^\infty\ker(P_\sM T^n)=\{0\}. $$
Hence, since
\[
\sD_T\ominus\{\cspan\{T^nD_T\sM,\;n\in\dN_0\}\}=\bigcap\limits_{n=0}^\infty\ker(P_\sM T^nD_T),
\]
we get $\cspan\{T^nD_T\sM,\;n\in\dN_0\}=\sD_T.$ This means that the operator ${\bf T}$ is $\sM$-simple.

Let
\[
\underline{}\dT=\begin{bmatrix}-P_\sM {\bf T}\uphar\sM&P_\sM D_{\bf T}\uphar\sD_{\bf T}\cr D_{\bf T}\uphar\sM &{\bf T}\uphar\sD_{\bf T} \end{bmatrix}=\begin{bmatrix}P_\sM {T}\uphar\sM&P_\sM D_{\bf T}\cr D_{\bf T}\uphar\sM &{\bf T}\uphar\sD_{\bf T}\end{bmatrix} :\begin{array}{l}\sM\\\oplus\\\sD_{\bf T}\end{array}\to\begin{array}{l}\sM\\\oplus\\\sD_{\bf T}\end{array}.
\]
As has been proved above because the
selfadjoint contraction ${\bf T}$ realizes the function $Q(\lambda):=(\lambda^2-1)^{-1}M(\lambda)^{-1}$, i.e.,
 \[
 P_\sM({\bf T}-\lambda I)^{-1}\uphar\sM=Q(\lambda)=\cfrac{M(\lambda)^{-1}}{\lambda^2-1},\; \lambda\in\dC\setminus [-1,1],
 \]
the selfadjoint contraction $\dT$ realizes the function $(\lambda^2-1)^{-1}Q(\lambda)^{-1}=M(\lambda)$. In addition, if $T$ is $\sM$-simple, then ${\bf T}$ and therefore ${\dT}$ are $\sM$-simple.
Since
\[
P_\sM({\dT}-\lambda I)^{-1}\uphar\sM=P_\sM({T}-\lambda I)^{-1}\uphar\sM=M(\lambda),\;|\lambda|>1,
\]
the operators $\dT$ and $T$ are unitarily equivalent and, moreover, see Theorem \ref{BROD}, there exists a unitary operator ${\dU}$ of the form
 \[
 \dU=\begin{bmatrix} I_\sM&0\cr 0& U\end{bmatrix}:\begin{array}{l}\sM\\\oplus\\\sD_{{\bf T}}\end{array}\to \begin{array}{l}\sM\\\oplus\\\cK\end{array},
 \]
 where $\cK:=\sH\ominus\sM$ and $U$ is a unitary operator from $\sD_{T}$ onto $\cK$ such that
\begin{multline*}
T\dU=\dU\dT\Longleftrightarrow\begin{bmatrix}P_\sM {T}\uphar\sM&P_\sM {T}\uphar\cK\cr P_\cK{T}\uphar\sM &P_\cK{T}\uphar\cK \end{bmatrix}\begin{bmatrix} I_\sM&0\cr 0& U\end{bmatrix}=\begin{bmatrix} I_\sM&0\cr 0& U\end{bmatrix}\begin{bmatrix}P_\sM {T}\uphar\sM&P_\sM D_{\bf T}\uphar\sD_{\bf T}\cr D_{\bf T}\uphar\sM &{\bf T} \end{bmatrix}\\
 \Longleftrightarrow \left\{\begin{array}{l}\left(P_\sM T\uphar\cK\right)U=P_\sM D_{\bf T}\uphar\sD_{\bf T}\\
P_\cK T\uphar\sM=UD_{\bf T}\uphar\sM\\
\left(P_\cK T\uphar\cK\right)U=U{\bf T}\uphar\sD_{\bf T}\uphar\sD_{\bf T}\end{array}\right..
\end{multline*}
In particular $P_\cK T\uphar\cK$ and ${\bf T}\uphar\sD_{\bf T}$ are unitarily equivalent.
\end{proof}
Observe that for a bounded selfadjoint $T$ the equality $M(\lambda)=P_\sM(T-\lambda I)^{-1}\uphar\sM$  yields the following relation for $\lambda\in\cmr$:
\[
\cfrac{1-|\lambda|^2}{\IM\lambda}\,\IM M(\lambda)-2\RE\left(\lambda M(\lambda)\right)-I_\sM=P_\sM (T-\lambda I)^{-1}(I-T^2)(T-\bar{\lambda}I)^{-1}\uphar\sM.
\]
Hence for $M(\lambda)\in {\bf N}^{0}_\sM [-1,1]$ we get
\[
\cfrac{1-|\lambda|^2}{\IM\lambda}\,\IM M(\lambda)-2\RE\left(\lambda M(\lambda)\right)-I_\sM =\cfrac{\IM\left((1-\lambda^2)M(\lambda)-\lambda\right)}{\IM\lambda}\ge 0,\; \IM\lambda\ne 0.
\]
\subsection{The fixed point of the mapping ${\bf \Gamma}$}

\begin{proposition}
\label{fixpoint} Let $\sM$ be a Hilbert space. Then
the
mapping ${\bf \Gamma}$ \eqref{GTR1} has a unique fixed point
\begin{equation}\label{fixfunczw}
M_0(\lambda)=-\cfrac{I_\sM}{\sqrt{\lambda^2-1}}\quad (\IM\sqrt{\lambda^2-1}> 0\quad\mbox{for}\quad\IM\lambda>0).
\end{equation}
Define the weight $\rho_0(t)$ and the weighted Hilbert space $\sH_0$ as follows
\begin{equation}\label{vesighj} \begin{array}{l}\rho_0(t)=\cfrac{1}{\pi}\cfrac{1}{\sqrt{1-t^2}},\; t\in (-1,1),\\
\sH_0:=L_2([-1,1],\sM,\rho_0(t))=L_2\left([-1,1],\;\rho_0(t)\right)\bigotimes\sM=\left\{f(t):\int\limits_{-1}^1\cfrac{||f(t)||^2_\sM}{\sqrt{1-t^2}}\, dt<\infty\right\}.
\end{array}
\end{equation}
Then $\sH_0$ is the Hilbert space with the inner product
$$\left(f(t),g(t)\right)_{\sH_0}=\cfrac{1}{\pi}\int\limits_{-1}^1(f(t),g(t))_\sM\,\rho_0(t)\,dt=\cfrac{1}{\pi}
\int\limits_{-1}^1\cfrac{(f(t),g(t))_\sM}{\sqrt{1-t^2}}\, dt.$$
Identify $\sM$ with a subspace of $\sH_0$ of constant vector-functions $\{f(t)\equiv f,\;f\in\sM\}$.
Define in $\sH_0$ the multiplication operator
\begin{equation}\label{ogthevy}
(T_0f)(t)=tf(t),\; f\in\sH_0.
\end{equation}
Then
\[
M_0(\lambda)=P_{\sM}(T_0-\lambda I)^{-1}\uphar\sM.
\]
 Let ${\bf H}_0=\bigoplus\limits_{j=0}^\infty\sM=\ell^2(\dN_0)\bigotimes\sM$ and let ${\bf J_0}$ be the operator in ${\bf H}_0$ given by the block-operator Jacobi matrix of the form \eqref{chebb1}.
 Set $\sM_0:=\sM\bigoplus\{0\}\bigoplus\{0\}\bigoplus\cdots$.
Then
\[
M_0(\lambda)=P_{\sM_0}({\bf J}_0-\lambda I)^{-1}\uphar{\sM}_0.
\]
\end{proposition}

\begin{proof}
Let $M_0(\lambda)$ be a fixed point of the mapping ${\bf \Gamma}$ , i.e.,
\[
M_0(\lambda)=\cfrac{M_0(\lambda)^{-1}}{\lambda^2-1}\Longleftrightarrow M_0(\lambda)^2=\cfrac{1}{\lambda^2-1}\,I_\sM,\;\lambda\in\dC\setminus[-1,1]
\]
Since $M_0(\lambda)$ is Nevanlinna function, we get \eqref{fixfunczw}.

For each $h\in\sM$ calculations give the equality, see \cite[pages 545--546]{Ber}, \cite{CReml},
\[
-\cfrac{h}{\sqrt{\lambda^2-1}}=\cfrac{1}{\pi}\int\limits_{-1}^1\cfrac{h}{t-\lambda}\,\cfrac{1}{\sqrt{1-t^2}}\,dt,\;\lambda\in\dC\setminus[-1,1].
\]
Therefore, if $T_0$ is the operator of the form \eqref{ogthevy}, then
\[
M_0(\lambda)=P_{\sM}(T_0-\lambda I)^{-1}\uphar\sM,\;\lambda\in\dC\setminus[-1,1].
\]
As it is well known the Chebyshev polynomials of the first kind
\[
\wh T_0(t)=1,\; \wh T_n(t):=\sqrt{2}\cos(n\arccos t),\;n\ge 1
\]
form an orthonormal basis of the space $L_2([-1,1],\rho_0(t)),$ where $\rho_0(t)$ is given by \eqref{vesighj}. This polynomials satisfy the recurrence
relations
\begin{multline*}
t\wh T_0(t)=\cfrac{1}{\sqrt{2}}\wh T_1(t),\;t\wh T_1(t)=\cfrac{1}{\sqrt{2}}\wh T_0(t)+\cfrac{1}{2}\wh T_2(t),\\
t\wh T_n(t)=\cfrac{1}{2}\wh T_{n-1}(t)+\cfrac{1}{2}\wh T_{n+1}(t),\;n\ge 2.
\end{multline*}
Hence the matrix of the operator $\sT_0$ of multiplication on the independent variable in the Hilbert space $L_2(\rho(t),[-1,1])$ w.r.t. the basis $\{\wh T_n(t)\}_{n=0}^\infty$ (the Jacobi matrix) takes the form \eqref{chebb1} when $\sM=\dC.$
Besides $m_0(\lambda):=(({{\bf J}_0}-\lambda I)^{-1}\delta_0,\delta_0)=-\cfrac{1}{\sqrt{\lambda^2-1}},$ where
$ \delta_0=\begin{bmatrix}1&0&0&\cdots\end{bmatrix}^T$ \cite{Ber}.
Since $T_0=\sT_0\bigotimes I_\sM$ we get that $T_0$ is unitarily equivalent to ${\bf J_0}=J_0\bigotimes I_\sM$ and
$M_0(\lambda)=P_{\sM_0}({\bf J}_0-\lambda I)^{-1}\uphar\sM_0$.
\end{proof}

Observe that $\sM$-valued holomorphic in $\dC\setminus [-1,1]$ function
 \[
M_1(\lambda):=2(-\lambda I_\sM-M^{-1}_0(\lambda))=2(-\lambda +{\sqrt{\lambda^2-1}})I_\sM
\]
belongs to the class ${\bf N}_{\sM}^0[-1,1].$
\section{The fixed point of the mapping ${\bf \wh \Gamma}$} \label{ffixx}

Now we will study the mapping ${\bf \wh\Gamma}$ \eqref{denken}.
 Let $\cM$ be a Nevanlinna family in the Hilbert space $\sM$. Then since
\[
\left|\IM(\left(\cM(\lambda)+\lambda I_\sM)f,f\right)\right|\ge |\IM\lambda|||f||^2,\; \IM\lambda\ne 0,\;f\in\dom\cM(\lambda),
\]
the estimate
\begin{equation} \label{jwtyrf}
||(\cM(\lambda)+\lambda I_\sM)^{-1}||\le \cfrac{1}{|\IM\lambda|},\;\IM\lambda\ne 0.
\end{equation}
holds true.
It follows that $\cM_1(\lambda)=-(\cM(\lambda)+\lambda I_\sM)^{-1}$ is $\bB(\sM)$-valued Nevanlinna function from the class $\cR_0[\sM]$ and, moreover,
$\cM_1(\lambda)=K^*(\wt T-\lambda I)^{-1}K,$ $\IM\lambda\ne 0$, where $\wt T$ is a selfadjoint operator in a Hilbert space $\wt\sH$  and $K\in\bB(\sM,\wt \sH)$ is a contraction, see Corollary \ref{ccorr1} and Proposition \ref{opexpr12}. For $\cM_2(\lambda)=-(\cM_1(\lambda)+\lambda I_\sM)^{-1}$
one has
\[
\lim\limits_{y\to\pm\infty}||iy\cM_2(iy)+I_\sM||=0,
\]
i.e., $\cM_2(\lambda)\in{\cN}[\sM].$ Thus, see Corollary \ref{ccorr1},
\begin{multline*}
\ran{\bf\wh\Gamma}={\bf\wh\Gamma}(\wt R[\sM])=\left\{M(\lambda)\in\cR_0[\sM]: s-\lim\limits_{y\to+\infty}\left(-iy M(iy)\right)\in [0,I_\sM]\right\},\\ \ran{\bf\wh\Gamma}^k\subset{\cN}[\sM],\;k\ge 2.
\end{multline*}

\begin{theorem}
\label{fixpunkt} Let $\sM$ be a Hilbert space. Then
\begin{enumerate}
\item 
the function
\begin{equation} \label{chebysh11}
\cM_0(\lambda)=\cfrac{-\lambda+\sqrt{\lambda^2-4}}{2}\,I_\sM,\;\IM \lambda\ne 0,\; \cM_0(\infty)=0
\end{equation}
 is a unique fixed point of the mapping ${\bf\wh \Gamma}$ \eqref{denken};

\item 
if ${\bf \wh\Gamma}(\cM)=\cM_0$, then $\cM(\lambda)=\cM_0(\lambda)$ for all $\lambda\in\cmr$;

\item
for every sequence of iterations of the form
 $$\cM_1(\lambda)=-(\cM(\lambda)+\lambda I_\sM)^{-1},\;\cM_{n+1}(\lambda)=-(\cM_{n}(\lambda)+\lambda I_\sM)^{-1},\;n=1,2\ldots,$$
 where $\cM(\lambda)$ is an arbitrary Nevanlinna function, the relation
 \[
 \lim\limits_{n\to\infty}||\cM_n(\lambda)-\cM_0(\lambda)||=0
 \]
 holds uniformly on each compact subsets of the open upper/lower half-plane of the complex plane $\dC$;

\item
the function $\cM_0(\lambda)$ is a unique fixed point for each degree of ${\bf\wh\Gamma}$.
\end{enumerate}

 \end{theorem}
\begin{proof}
(1)  Since
\[
\cM(\lambda)=-(\cM(\lambda)+\lambda I_\sM)^{-1}\Longleftrightarrow \cM^2(\lambda)+\lambda\cM(\lambda)+I_\sM=0,
\]
and $\cM$ is a Nevanliina family, we get
that $\cM_0$ given by \eqref{chebysh11} is a unique solution.

(2) Suppose ${\bf \wh\Gamma}(\cM)=\cM_0,$ i.e.,
\[
-(\cM(\lambda)+\lambda I_\sM)^{-1}=\cfrac{-\lambda+\sqrt{\lambda^2-4}}{2}\,I_\sM,\; \lambda\in\cmr.
\]
Then
\[
\cM(\lambda)=\left(-\cfrac{2}{-\lambda+\sqrt{\lambda^2-4}}-\lambda\right)I_\sM=\cfrac{-\lambda+\sqrt{\lambda^2-4}}{2}\,I_\sM=\cM_0(\lambda).
\]

(3)
Let $\cF$
 and $\cG$ be two $\bB(\sM)$-valued Nevanlinna functions. Set
$$\wh F(\lambda)=-(\cF(\lambda)+\lambda I_\sM )^{-1},\; \wh G(\lambda)=-(\cG(\lambda)+\lambda I_\sM)^{-1},\; \lambda\in \cmr .$$
Then $\wh F$ and $\wh G$ are $\bB(\sM)$-valued and
\[
\wh F(\lambda)-\wh G(\lambda)=(\cF(\lambda)+\lambda I_\sM )^{-1}\left(\cF(\lambda)-\cG(\lambda)\right)(\cG(\lambda)+\lambda I_\sM )^{-1}.
\]
From \eqref{jwtyrf} we get
\[
||(\wh F(\lambda)-\wh G(\lambda))||\le \cfrac{1}{|\IM\lambda|^2}||\cF(\lambda)-\cG(\lambda)||.
\]
Hence for the sequence of iterations $\{\cM_n(\lambda)\}$ one has
\[
||(\cM_n(\lambda)-\cM_m(\lambda))||\le \cfrac{1}{(|\IM\lambda|^2)^{m-1}}||\cM_{n-m+1}(\lambda)-\cM_1(\lambda)||,\; n>m.
\]
It follows that if $|\IM\lambda|>1$, then
\[
||( \cM_n(\lambda)-\cM_m(\lambda))||\le \cfrac{(|\IM\lambda|^2)^{-m+1}}{1-(|\IM \lambda|)^{-2}}||\cM_{2}(\lambda)-\cM_1(\lambda)||,\; n>m.
\]
Therefore, the sequence of linear operators $\{\cM_n(\lambda)\}_{n=1}^\infty$ convergence in the operator norm topology, and the limit satisfies the equality
$\cM(\lambda)=-(\cM(\lambda)+\lambda I)^{-1}$, i.e., is the fixed point of the mapping ${\bf\wh\Gamma}$. In addition due to the inequality
\[
||(\cM_n(\lambda)-\cM_m(\lambda))||\le \cfrac{1}{R^{m-1}}||\cM_{n-m+1}(\lambda)-\cM_1(\lambda)||,\; n>m, \; |\IM \lambda|\ge R,\; R>1
\]
 we get that the convergence is uniform on $\lambda$ on the domain $\{\lambda:|\IM \lambda|\ge R\}$, $R>1$.

 Note that from
 \[
 ||\cM_n(\lambda)||=||(\cM_{n-1}(\lambda)+\lambda I_\sM)^{-1}||\le \cfrac{1}{|\IM\lambda|},\;\IM\lambda\ne 0
 \]
 it follows that the sequence of operator-valued functions $\{\cM_n(\lambda)\}_{n=1}^\infty$ is uniformly bounded on $\lambda$ on each domain $|\IM\lambda|>r$, $r>0$.
 Thus, the sequence $\{\cM_n\}_{n=1}^\infty$ is locally uniformly bounded in the upper and lower open half-planes and, in addition, $\{\cM_n\}$
 uniformly converges in the operator-norm topology on the domains $\{\lambda:|\IM \lambda|\ge R\}$, $R>1$. By the Vitali-Porter theorem \cite{Schiff} the relation
 \[
 \lim\limits_{n\to\infty}||\cM_n(\lambda)-\cM_0(\lambda)||=0
 \]
 holds uniformly on $\lambda$ on each compact subset of the open upper/lower half-plane of the complex plane $\dC$.

 (4) The function $\cM_0$ is a fixed point for each degree of ${\bf\wh\Gamma}$. Suppose that the mapping ${\bf\wh\Gamma}^{l_0}$, $l_0\ge 2$ has
 one more fixed point $\cL_0(\lambda)$. Then arguing as above, we get
 \[
 ||\cM_0(\lambda)-\cL_0(\lambda)||\le|\IM\lambda|^{-2l_0} ||\cM_0(\lambda)-\cL_0(\lambda)||\; \forall\lambda\in\cmr.
 \]
 It follows that $\cL_0(\lambda)\equiv\cM_0(\lambda)$.
 \end{proof}

The scalar case ($\sM=\dC$) of the next Proposition can be found in \cite[pages 544--545]{Ber}, \cite{CReml}.
 \begin{proposition}\label{gjckt} Let $\sM$ be a Hilbert space.
\begin{enumerate}
\item Consider the weighted Hilbert space
$$\sL_0:=L_2\left([-2, 2],\;\cfrac{1}{2\pi}\sqrt{4-t^2}\right)\otimes\sM $$ and the operator
\[
({\cT}_0 f)(t)=tf(t),\; f(t)\in\sL.
\]
Identify $\sM$ with a subspace of $\sL_0$ of constant vector-functions $\{f(t)\equiv f,\;f\in\sM\}$. Then
\[
\cM_0(\lambda)=P_{\sM}(\cT_0-\lambda I)^{-1}\uphar\sM, \;\lambda\in \dC\setminus [-2,2],
\]
where $\cM_0(\lambda)$ is given by \eqref{chebysh11}.

 \item Let ${\bf H}_0=\bigoplus\limits_{j=0}^\infty\sM=\ell^2(\dN_0)\bigotimes\sM$ and let ${\bf \wh J_0}$ be the operator in ${\bf H}_0$ given by the block-operator Jacobi matrix of the form \eqref{clschope}.

Set $\sM_0:=\sM\bigoplus\{0\}\bigoplus\{0\}\bigoplus\cdots$.
Then $$
\cM_0(\lambda)=P_{\sM_0}({\bf \wh J}_0-\lambda I)^{-1}\uphar\sM_0,\;\lambda\in \dC\setminus [-2,2].
$$
\end{enumerate}
\end{proposition}
In the next statement we show that one can construct a sequence $\{\wh \sH_n,\wh A_n\}$ of realizations for the iterates $\{\cM_{n+1}={\bf \wh\Gamma}(\cM_n)\}_{n=1}^\infty$ that inductively converges to $\{{\bf H_0},{\bf\wh J}_0\}.$
\begin{theorem} \label{iteraw}
Let $\cM(\lambda)$ be an arbitrary Nevanlinna family in $\sM$.
Define the iterations of the mapping ${\bf\wh \Gamma}$ \eqref{denken}:
\begin{multline*}
\cM_1(\lambda)=-(\cM(\lambda)+\lambda I_\sM)^{-1},\;
\cM_{n+1}(\lambda)=-(\cM_{n}(\lambda)+\lambda I_\sM)^{-1},\;n=1,2\ldots,
\\\lambda\in\cmr.
\end{multline*}
Let $\cM_1(\lambda)=K^*(\wh T-\lambda I)^{-1}K,$ $\IM\lambda\ne 0$ be a realization of $
\cM_1(\lambda)$, where $\wh T$ is a selfadjoint operator in the Hilbert space $\wh\sH$ and $K\in\bB(\sM, \wh\sH)$ is a contraction.
 Further, set
\begin{multline}\label{chaspaces}
\wh \sH_1=\sM\oplus\wh\sH,\;\wh\sH_2=\sM\oplus\wh\sH_1=\sM\oplus\sM\oplus\wh\sH,\\
\wh\sH_{n+1}=\sM\oplus\sH_n=\underbrace{\sM\oplus\sM\oplus\cdots\oplus\sM}_{n+1}\oplus\wh \sH,\ldots
\end{multline}
and define the following linear operators for each $n\in\dN$:
\[
\begin{array}{l}
\sM\ni x\mapsto\dI^{(n)}_\sM x=[x,\underbrace{0,0,\ldots,0}_{n}]^T\in\wh\sH_n,\\
\wh\sH_n\ni\begin{bmatrix}x\cr h\end{bmatrix}\mapsto P^{(0,n)}_\sM \begin{bmatrix}x\cr h\end{bmatrix}=x\in\sM(\perp\wh\sH_n)\;\forall x\in\sM, \;\forall h\in\wh\sH_n.
\end{array}
\]
Define selfadjoint operators in the Hilbert spaces $\wh\sH_n$ for $n\in\dN$:
\begin{multline}\label{chaoper}
\wh A_1=\begin{bmatrix} 0&K^*\cr K&\wh T\end{bmatrix}:\begin{array}{l}\sM\\\oplus\\\wh \sH\end{array}\to
\begin{array}{l}\sM\\\oplus\\\wh \sH\end{array},\;\dom\wh A_1=\sM\oplus\dom\wh T,\\
\wh A_2=\begin{bmatrix}0&P^{(0,1)}_\sM\cr \dI^{(1)}_\sM&\wh A_1 \end{bmatrix}:\begin{array}{l}\sM\\\oplus\\\wh \sH_1\end{array}\to
\begin{array}{l}\sM\\\oplus\\\wh \sH_1\end{array},\; \dom \wh A_2=\sM\oplus\dom\wh A_1\\
\wh A_{n+1}=\begin{bmatrix}0&P^{(0,n)}_\sM\cr \dI^{(n)}_\sM&\wh A_n \end{bmatrix}:\begin{array}{l}\sM\\\oplus\\\wh \sH_n\end{array}\to
\begin{array}{l}\sM\\\oplus\\\wh \sH_n\end{array},\;
\dom \wh A_{n+1}=\sM\oplus\dom \wh A_n
\end{multline}

 Then $\wh A_n$ is a realization of $\cM_{n+1}$ for each $n$, i.e.,
\begin{equation}
\label{realmnl}
\cM_{n+1}(\lambda)=P_\sM(\wh A_n-\lambda I)^{-1}\uphar\sM,\;\;n=1,2\ldots,\;\lambda\in\cmr.
\end{equation}
If $\wh T$ is $\cran K$-simple, i.e., $\cspan\{(\wh T-\lambda)^{-1}\ran K:\;\lambda\in\cmr\}=\cK$, then $\wh A_n$ is $\sM$-minimal  for each $n\in\dN.$
Moreover, the Hilbert space ${\bf H}_0$ and the block-operator Jacobi matrix \eqref{clschope} are the inductive limits
${\bf H}_0=\lim\limits_{\rightarrow}\wh\sH_n$ and ${\bf\wh J}_0=\lim\limits_{\rightarrow}\wh A_n,$ of the chains $\{\wh\sH_n\}$ and $\{\wh A_n\}$, respectively. \end{theorem}
\begin{proof}
Relations in \eqref{realmnl} follow  by induction from \eqref{Sh-Fr2}.

Note that the operator $\wh A_n$ can be represented by the block-operator matrix
\begin{equation}\label{jamat}
{\wh A_n}=\begin{bmatrix} 0 & I_\sM & 0 &0   & 0 &
\cdot &\cdot &\cdot&0 \\
I_\sM & 0 & I_\sM & 0 &0& \cdot &
\cdot &\cdot&0 \\
0    & I_\sM & 0 & I_\sM &0& \cdot &
\cdot&\cdot&0   \\
0    & 0& I_\sM & 0 & I_\sM &0& \cdot &
\cdot&0
\\
\vdots & \vdots & \vdots & \vdots & \vdots & \vdots & \vdots&\vdots&\vdots\\
0&0&\cdot&\cdot&\cdot&0&0&I_\sM&0\\
0&0&\cdot&\cdot&\cdot&0&I_\sM&0&K^*\\
0&0&\cdot&\cdot&\cdot&0&0&K&\wh T
\end{bmatrix}:\begin{array}{l} n\left\{\begin{array}{l}\sM\\\oplus\\\sM\\\oplus\\\vdots\\\oplus\\\sM\end{array}\right.\\\qquad\oplus\\\qquad\;\wh\sH\end{array}\longrightarrow
\begin{array}{l} n\left\{\begin{array}{l}\sM\\\oplus\\\sM\\\oplus\\\vdots\\\oplus\\\sM\end{array}\right.\\\qquad\oplus\\\qquad\;\wh\sH\end{array}.
\end{equation}
Besides, if $\wh T$ is bounded, then all operators $\{\wh A_n\}_{n\ge 1}$ are bounded and each $\cM_n(\lambda)$ belongs to the class ${\bf N}^0_\sM$ for $n\ge 2$.

Define the linear operators $\gamma_k^l:\wh\sH_k\to\wh\sH_l$, $l\ge k$,  $\gamma_k:\wh\sH_k\to{\bf H_0},$ $k\in\dN$ as follows
\begin{multline}
\label{isomcha}
\gamma_k^l[f_1,f_2,\ldots, f_k,\varphi]=[f_1,f_2,\ldots, f_k,\underbrace{0,0,\ldots,0}_{l-k},\varphi],\\
\gamma_k[f_1,f_2,\ldots, f_k,\varphi]=[f_1,f_2,\ldots, f_k,0,0,\ldots],\\
\{f_i\}_{i=1}^k\subset\sM,\varphi\in\wh\sH.
\end{multline}
Then
\begin{enumerate}
\item $\gamma_k^k$ is the identity on $\wt \sH_k$ for each $k\in\dN$,
\item $\gamma_k^m=\gamma_l^m\circ\gamma_k^l$ if $k\le l\le m,$
\item $\gamma_k=\gamma_l\circ\gamma_k^l,$ $l\ge k,$ $k\in\dN,$
\item ${\bf H}_0=\cspan \{\gamma_k\wh \sH_k,\; k\ge 1\}.$
\end{enumerate}
Note that the operators $\{\gamma_k^l\}$ are isometries and the operators $\{\gamma_k\}$ are partial isometries and $\ker\gamma_k=\wt\sH$ for all $k$.
The family $\{\wh\sH_k,\gamma_k^l,\gamma_k\}$ forms the inductive isometric chain \cite{Marchenko} and the Hilbert space ${\bf H}_0$ is the inductive limit of the Hilbert spaces $\{\wh\sH_n\}$ \eqref{chaspaces}: ${\bf H}_0=\lim\limits_{\rightarrow}\wh\sH_n.$

Define following \cite{Marchenko} on $\cD_\infty:=\bigcup\limits_{n=1}^\infty\gamma_n \dom \wh A_n$
a linear operator in ${\bf H}_0$:
\[
\wh A_\infty h:=\lim\limits_{m\to\infty}\gamma_m\wh A_m\gamma_k^m h_k,\; h=\gamma_k h_k,\;h_k\in\wh\sH_k\ominus\wh\sH,
\]
where $\{\wh A_n\}$ are defined in \eqref{chaoper}.
Due to \eqref{isomcha} and \eqref{jamat} the operator $\wh A_\infty$ exists, densely defined and its closure is bounded selfadjoint operator in ${\bf H}_0$ given by the block-operator matrix ${\bf\wh J_0}$ of the form \eqref{clschope}.

\end{proof}
Note that the operator ${\bf\wh J_0}$ is called the free discrete Schr\"{o}dinger operator \cite{CReml}.
 Observe also that the function
\[
M_1(\lambda)=\cfrac{1}{2}\cM_0\left(\cfrac{\lambda}{2}\right)=2(-\lambda +{\sqrt{\lambda^2-1}})I_\sM, \; \lambda\in\dC\setminus [-1,1],
\]
where $\cM_0(\lambda)$ is given by \eqref{chebysh11}, belongs to the class ${\bf N}_{\sM}^0[-1,1]$.
Besides, for all $\lambda\in \dC\setminus[-1,1]$ the equality $
M_1(\lambda)=P_\sM({\cT}_1-\lambda I)^{-1}\uphar\sM
$
holds, where $\cT_1$ is the multiplication operator $({\cT}_1 f)(t)=tf(t)
$
in the weighted Hilbert space
$$L_2\left([-1,1],\;\cfrac{2}{\pi}\sqrt{1-t^2}\right)\otimes\sM. $$
If $\sM=\dC,$ then the matrix of the corresponding operator $\cT_1$ in the orthonormal basis of the Chebyshev polynomials of the second kind
\[
U_n(t)=\cfrac{\sin[(n+1)\arccos t]}{\sqrt{1-t^2}},\;n=0,1,\ldots
\]
is of the form $\cfrac{1}{2}{\bf \wh J}_0$ \cite{Ber}.

\section{Canonical systems and the mapping ${\bf\wh\Gamma}$}\label{LSEC}
Let $m\in {\bf N}^0_\dC$. Then, see \cite[Chapter VII, $\S$1, Theorem 1.11]{Ber}, \cite{GS}, \cite{CReml}, the function $m$ is the compressed resolvent ($m(\lambda)=\left(\left( J-\lambda I\right)^{-1}\delta_0,\delta_0\right)$) of a unique finite or semi-infinite Jacobi matrix $J=J(\{a_k\},\{b_k\})$
with real diagonal entries $\{a_k\}$ and positive off-diagonal entries $\{b_k\}$ and in the semi-infinite case one has $\{a_k\},\{b_k\}\in\ell^\infty(\dN_0)$.
 Observe that the entries of $J$
can be found using the continued fraction (J-fraction) expansion of $m(\lambda)$ \cite{GS}, \cite{Wall}
\[
m(\lambda)= \frac{-1}{\lambda-a_0}\;\raisebox{-3mm}{{\rm
+}}\;\frac{-b^2_0}{\lambda-a_1}\;\raisebox{-3mm}{{\rm
+}}\;\frac{-b^2_1}{\lambda-a_2} \;\raisebox{-3mm}{{\rm +}\,\ldots}\;
\raisebox{-3mm}{{\rm +}}\;\frac{-b^2_{n-1}}
{\lambda-a_n}\;\raisebox{-3mm}{{\rm +}\,\dots}.
\]
On the other hand the algorithm of I.S.~Kac \cite{Kac1999} enables to construct for given $J(\{a_k\},\{b_k\})$ the Hamiltonian $\cH(t)$ such that the $m$-function of $J(\{a_k\},\{b_k\})$ is the $m$-function of the corresponding canonical system of the form \eqref{cansys}.

Below we give the algorithm of Kac. Let $J$ be a semi-infinite Jacobi matrix
\begin{equation}\label{jacab}
J=J(\{a_k\},\{b_k\})=\begin{bmatrix}a_0 & b_0 & 0 &0   & 0 &
\cdot &\cdot &\cdot \\
b_0& a_1 & b_1 & 0 &0& \cdot &
\cdot &\cdot \\
0    & b_1 & a_2 & b_2 &0& \cdot &
\cdot&\cdot   \\
\vdots & \vdots & \vdots & \vdots & \vdots & \vdots & \vdots&\vdots\end{bmatrix}.
\end{equation}
The condition $\{a_k\},\{b_k\}\in\ell^\infty(\dN_0)$ is necessary and sufficient for the boundedness of the corresponding selfadjoint operator in the Hilbert space $\ell^2(\dN_0)$.

Put
\begin{equation}\label{Kac1}
l_{-1}=1, l_0=1, \;\theta_{-1}=0,\; \theta_0=\cfrac{\pi}{2}.
\end{equation}
  Then calculate
\begin{equation}\label{Kac2}
\begin{array}{l}
\theta_{1}=\arctan a_0+\pi,\;l_1=\cfrac{1}{l_0b^2_0\sin^2(\theta_1-\theta_0)}.\\
\end{array}
\end{equation}
Find $\theta_2$ from the system
\begin{equation}\label{Kac3}
\left\{\begin{array}{l}\cot(\theta_2-\theta_1)=-a_1l_1-\cot(\theta_1-\theta_0)\\
\theta_2\in(\theta_1,\theta_1+\pi)\end{array}\right..
\end{equation}
Find successively $l_j$ and $\theta_{j+1}$, $j=2,3,\ldots$
\begin{equation}
\label{Kac4}
\begin{array}{l}
l_j=\cfrac{1}{l_{j-1}b^2_{j-1}\sin^2(\theta_j-\theta_{j-1})},\\
\left\{\begin{array}{l}\cot(\theta_{j+1}-\theta_j)=-a_jl_j-\cot(\theta_j-\theta_{j-1})\\
\theta_{j+1}\in(\theta_j,\theta_j+\pi)\end{array}\right.
\end{array}.
\end{equation}
Define intervals $[t_j,t_{j+1})$ as follows
\begin{multline}\label{Kac5}
t_{-1}=-1,\; t_0=t_{-1}+l_{-1}=0,\;t_1=t_0+l_0=1,
 \;t_{j+1}=t_j+l_j=1+\sum_{k=1}^j l_k,\;j\in\dN.
\end{multline}
Then necessarily, \cite{Kac1999}, we get that $\lim_{j\to\infty}t_j=+\infty.$
Finally define the right continuous increasing step-function
\begin{equation}\label{Kac6}
\theta(t):=\left\{\begin{array}{l}\theta_0=\cfrac{\pi}{2},\; t\in(t_0,t_1)=(0,1)\\
\theta_j,\;t\in[t_j, t_{j+1}),\;j\in\dN\end{array}\right.
\end{equation}
and the Hamiltonian $\cH(t)$ on $\dR_+$
\begin{multline}\label{Kac7}
\cH(t):=\begin{bmatrix}\cos\theta(t)\cr\sin\theta(t)\end{bmatrix}\begin{bmatrix}\cos\theta(t)&\sin\theta(t)\end{bmatrix}=
\begin{bmatrix}\cos^2\theta(t)&\cos\theta(t)\sin\theta(t)\cr \cos\theta(t)\sin\theta(t)&\sin^2\theta(t)\end{bmatrix}\\
=\cfrac{1}{2}\begin{bmatrix}1&0\cr0&1\end{bmatrix}+\cfrac{1}{2}\begin{bmatrix}\cos 2\theta(t)&\sin 2\theta(t)\cr\sin 2\theta(t)&-\cos 2\theta(t)\end{bmatrix}.
\end{multline}
Then the Nevanlinna function $m(\lambda)=((J-\lambda I)^{-1}\delta_0,\delta_0)$ coincides with $m$-function of the corresponding canonical system of the form \eqref{cansys}.
 Observe that the algorithm shows that
 \begin{equation}\label{vsegda}
 \cH(t)=\begin{bmatrix}0&0\cr 0&1\end{bmatrix},\; t\in[0,1).
\end{equation}
 Using \eqref{Kac1}--\eqref{Kac7} for the Jacobi matrix $\wh J_0$
\[
\wh J_0=\begin{bmatrix} 0 & 1 & 0 &0   & 0 &
\cdot &
\cdot &\cdot \\
1 & 0 & 1 & 0 &0& \cdot &
\cdot &\cdot \\
0    & 1 & 0 & 1 &0& \cdot &
\cdot&\cdot   \\
\vdots & \vdots & \vdots & \vdots & \vdots & \vdots & \vdots&\vdots
\end{bmatrix},
\]
we get
$$l^{0}_j=1,\;\;\theta^{0}_j=(j+1)\cfrac{\pi}{2}\qquad \forall j\in\dN_0,$$
\[
\theta^{0}(t)=(j+1)\cfrac{\pi}{2},\;t\in[j, j+1)\;\forall j\in\dN_0,
\]
\begin{multline}\label{Hamiltoo}
\cH_0(t)=\begin{bmatrix}\cos^2(j+1)\cfrac{\pi}{2}&0\cr 0&\sin^2(j+1)\cfrac{\pi}{2}\end{bmatrix}\\
=\cfrac{1}{2}\begin{bmatrix}1-(-1)^j&0\cr 0&1+(-1)^j\end{bmatrix},\; t\in [j,j+1)\;\forall j\in\dN_0.
\end{multline}

\begin{proposition}
Let the scalar non-rational Nevanlinna function $m$ belong to the class ${\bf N}^0_\dC$. Define the functions
\[
m_1(\lambda)=-\cfrac{1}{m(\lambda)+\lambda}\,,\ldots, m_{n+1}(\lambda)=-\cfrac{1}{m_n(\lambda)+\lambda},\ldots,\; \lambda\in\cmr.
\]
Let $J$ be the Jacobi matrix with the $m$-function $m$, i.e., $m(\lambda)=\left(\left( J-\lambda I\right)^{-1}\delta_0,\delta_0\right),$
$\forall\lambda\in\cmr$.
Assume that $\cH(t)$ is the Hamiltonian such that the $m$-function of the corresponding canonical system coincides with $m$.
Then the Hamiltonian $\cH_n(t)$ of the canonical system whose $m$-function coincides with $m_n,$ takes the form
\begin{multline}\label{Hamilo}
\cH_n(t)=\left\{\begin{array}{l}\cH_0(t),\; t\in[0,n+1),\\
(-1)^{n}\cH(t-n)+\cfrac{1}{2}\begin{bmatrix}1-(-1)^n&0\cr 0&1-(-1)^n\end{bmatrix},\;t\in[n+1,\infty)\end{array}\right.\\
=\left\{\begin{array}{l}\cH_0(t),\; t\in[0,n+1),\\
\begin{bmatrix}\cos^2\left(\theta_j+n\cfrac{\pi}{2}\right)&\cfrac{(-1)^n}{2}\sin 2\theta_j\cr \cfrac{(-1)^n}{2}\sin 2\theta_j&\sin^2\left(\theta_j+n\cfrac{\pi}{2}\right)\end{bmatrix},\; t\in [t_{j}+n,t_{j+1}+n),\;j\in\dN
\end{array}\right.,
\end{multline}
where $\{t_j,\theta_j\}_{j\ge 1}$ are parameters of the Hamiltonian $\cH(t)$.
\end{proposition}
\begin{proof}
Set
\begin{equation}\label{seqjac}
J_1=\left[\begin{array}{c|c}0&\begin{array}{cccc}1&0&0&\ldots
\end{array}\\
\hline\begin{array}{c}1\cr
0\cr\vdots\end{array}&J\end{array}\right],\ldots, J_n=\left[\begin{array}{c|c}0&\begin{array}{cccc}1&0&0&\ldots
\end{array}\\
\hline\begin{array}{c}1\cr
0\cr\vdots\end{array}&J_{n-1}\end{array}\right],\ldots.
\end{equation}
Then \eqref{Sh-Fr2} and induction yield the equalities
\begin{multline*}
\left((J_1-\lambda I)^{-1}\delta_0,\delta_0\right)=-(m(\lambda)+\lambda)^{-1}=m_1(\lambda),\ldots,\\
\left((J_{n}-\lambda I)^{-1}\delta_0,\delta_0\right)=-(m_{n-1}(\lambda)+\lambda)^{-1}=m_n(\lambda),\dots,
\;\lambda\in\cmr.
\end{multline*}
Let $J=J\left(\{a_k\}_{k=0}^\infty,\{b_k\}_{k=0}^\infty\right)$ be of the form \eqref{jacab}. Then from \eqref{seqjac}
it follows that for the entries of $J_n=J_n\left(\{a^{(n)}_{k}\}_{k=0}^\infty,\{b^{(n)}_{k}\}_{k=0}^\infty\right),$ $n\in\dN,$ we have the equalities
\begin{equation}\label{entrn}
\left\{\begin{array}{l} a^{(n)}_0=a^{(n)}_1=\cdots=a^{(n)}_{n-1}=0\\
a^{(n)}_{k}=a_{k-n},\; k\ge n\end{array}\right.,\; \left\{\begin{array}{l} b^{(n)}_0=b^{(n)}_1=\cdots=b^{(n)}_{n-1}=1\\
b^{(n)}_{k}=b_{k-n},\; k\ge n\end{array}\right..
\end{equation}
In order to find an explicit form of the Hamiltonian corresponding to the Nevanlinna function $m_n$ we apply the algorithm of Kac described
by \eqref{Kac1}, \eqref{Kac2}, \eqref{Kac3}, \eqref{Kac4}, \eqref{Kac5}, \eqref{Kac6}, \eqref{Kac7}.
Then we obtain
\[
\begin{array}{l} l^{(n)}_{-1}=l^{(n)}_0=l^{(n)}_1=\ldots=l^{(n)}_{n}=1,\\
\theta^{(n)}_{-1}=0,\; \theta^{(n)}_0=\cfrac{\pi}{2},\; \theta^{(n)}_1=\pi,\ldots, \theta^{(n)}_{n}=(n+1)\cfrac{\pi}{2},\\
l^{(n)}_{n+j}=l_{j},\; \theta^{(n)}_{n+j}=\theta_{j}+(n+2)\cfrac{\pi}{2},\; j\in\dN.
\end{array}
\]
Hence \eqref{Kac7} and \eqref{Hamiltoo} yield \eqref{Hamilo}.
\end{proof}

By Theorem \ref{fixpunkt} the sequence $\{m_n\}$ of Nevanlinna functions converges uniformly on each compact subset of $\dC_+/\dC_-$ to the Nevanlinna function
\[
m_0(\lambda)=\cfrac{-\lambda+\sqrt{\lambda^2-4}}{2},\;\lambda\in\cmr.
\]
This function is the $m$-function of the Jacobi matrix $\wh J_0$ and the $m$-function of the canonical system with the Hamiltonian $\cH_0$.
From \eqref{seqjac} we see that for the  sequence of selfadjoint  Jacobi operators $\{J_n\}$ in $\ell^2(\dN_0)$ the relations
\[
P_nJ_{n+1}P_n =P_nJ_0P_n\qquad \forall n\in\dN_0
\]
 hold,
where $P_n$ is the orthogonal projection in $\ell^2(\dN_0)$ on the subspace
\[
E_n=\span\{\delta_0, \delta_1,\ldots,\delta_{n-1}\}.
\]
It follows that
\[
s-\lim\limits_{n\to\infty}P_nJ_{n+1}P_n=\wh J_0.
\]
For the sequence \eqref{Hamilo} of $\{\cH_n\}$ one has
\begin{equation}
\label{posled} \cH_n\uphar[0,n+1)=\cH_0\uphar [0,n+1)\qquad\forall n.
\end{equation}
From \eqref{posled} it follows that if $\vec f(t)=\begin{bmatrix}f_1(t)\cr f_2(t)\end{bmatrix}$ is a continuous function on $\dR_+$ with a compact support, then there exists $n_0\in\dN$ such that
$\int\limits_0^\infty\vec f(t)^*\cH_n(t)\vec f(t)dt=\int\limits_0^\infty\vec f(t)^*\cH_0(t)\vec f(t)dt$ for all $n\ge n_0$.

 It is proved in \cite[Proposition 5.1]{Hur2016} that for a sequence of canonical systems with Hamiltonians $\{H_n\}$ and $H$ the convergence $m_{H_n}(\lambda)\to m_{H}(\lambda),$ $n\to\infty$ of $m$-functions holds locally uniformly on $\dC_+/\dC_-$ if and only if
$
\int\limits_0^\infty\vec f(t)^*H_n(t)\vec f(t)dt\to\int\limits_0^\infty\vec f(t)^*H(t)\vec f(t)dt
$
for all continuous functions $\vec f(t)$ with compact support on $\dR_+$.

In conclusion we note that the equalities \eqref{vsegda}, \eqref{Hamiltoo}, and \eqref{Hamilo} (for $n=1$) show that for the transformation ${\bf\wh\Gamma}$ one has the following scheme:
\begin{multline*}
{\bf N}^0_\dC\ni m\;\mbox{(non-rational)}\longrightarrow\cH(t)\Longrightarrow\\
\cH_{\bf\wh\Gamma}(t)=\left\{\begin{array}{l} \cH_0(t),\; t\in[0,2)\\
\begin{bmatrix}1&0\cr 0& 1\end{bmatrix}-\cH(t-1),\; t\in[2,+\infty)\end{array}\right.\longleftarrow{\bf\wh \Gamma}(m).
\end{multline*}

\end{document}